\documentclass[11pt,reqno]{amsart}

\usepackage{amssymb}
\usepackage{amscd}
\usepackage{amsfonts}
\usepackage[all]{xy}

\numberwithin{equation}{section} \theoremstyle{plain}
\newtheorem{theorem}{Theorem}[section]
\newtheorem{proposition}[theorem]{Proposition}
\newtheorem{lemma}[theorem]{Lemma}
\newtheorem{corollary}[theorem]{Corollary}
\newtheorem{definition}[theorem]{Definition}

\newtheorem*{mainthm3-repeat}{Theorem \ref{mainthm3}}
\newtheorem*{sec3-key-prop-rpt}{Theorem \ref{sec3-key-prop}}

\newsavebox{\proofbox}
\savebox{\proofbox}{\begin{picture}(7,7)  \put(0,0){\framebox(7,7){}}\end{picture}}

\newcommand\QQ{\mathbb{Q}}

\newcommand\FF{\mathbb{F}}
\newcommand\Z{\mathbb{Z}}
\newcommand\R{\mathbb{R}}

\newcommand\C{\mathbb{C}}
\newcommand\N{\mathbb{N}}
\newcommand\A{\mathbb{A}}
\newcommand\G{\mathbf{G}}
\newcommand\m{\mathfrak{m}}

\newcommand\cha{\operatorname{char}\,}

\newcommand\Aut{\operatorname{Aut}}
\newcommand\SL{\operatorname{SL}}
\newcommand\PSL{\operatorname{PSL}}
\newcommand\ASL{\operatorname{ASL}}
\newcommand\SO{\operatorname{SO}}

\newcommand\Br{\operatorname{Br}}

\newcommand\GL{\operatorname{GL}}
\newcommand\PGL{\operatorname{PGL}}

\newcommand\Rad{\operatorname{Rad}}

\renewcommand\ss{\operatorname{cr}}
\newcommand\rk{\operatorname{rk}}

\newcommand\Gal{\operatorname{Gal}}

\newcommand\Spec{\operatorname{Spec}\,}

\newcommand\Hom{\operatorname{Hom}}

\newcommand\Sp{\operatorname{Sp}}

\newcommand\F{\mathbb{F}}
\newcommand\Q{\mathbb{Q}}

\newcommand\cF{\mathcal{F}}

\def\m{\mathfrak{m}}

\begin{document}

\title[Strongly dense free subgroups II]{Strongly dense free subgroups of semisimple algebraic groups  II}
\author[Breuillard]{Emmanuel Breuillard}
\address{Mathematical Institute\\
University of Oxford\\
Woodstock Rd\\
Oxford OX2 6GG\\
United Kingdom}

\email{breuillard@maths.ox.ac.uk}

\author[Guralnick]{Robert Guralnick}
\address{Department of Mathematics \\
University of Southern California\\
3620 Vermont Avenue\\
Los Angeles, California 90089-2532}
\email{guralnic@usc.edu}

\author{Michael Larsen}
\address{Department of Mathematics,
Indiana University,
Bloomington, IN
U.S.A. 47405}
\email{mjlarsen@indiana.edu}

\thanks{%Emmanuel Breuillard  was supported in part by the ERC starting grant 208091-GADA.
Robert Guralnick was partially supported by the
NSF grant DMS-1901595 and Simons Foundation Fellowship 609771.
Michael Larsen was partially supported by the NSF grant DMS-2001349
and Simons Foundation Fellowship 917214
and would also like to acknowledge the hospitality of the IAS while some of the work on this paper was being done.}


\begin{abstract}
It was shown in \cite{BGGT} 
  that there exist \emph{strongly dense} free subgroups in any semisimple 
  algebraic group over a large enough field. These are nonabelian free subgroups 
  all of whose subgroups are either cyclic or Zariski-dense.  Here we show that the 
  same is true for as long as the transcendence degree of the field is at least 
  $1$ in characteristic $0$ and at least $2$ in positive characteristic. We also consider related questions for surface groups.
\end{abstract}

\maketitle

\setcounter{tocdepth}{1}
\tableofcontents

\section{Introduction}\label{sec1}

There is a long history regarding free subgroups of linear groups, in which J. Tits played a prominent role with his celebrated alternative \cite{tits72}: every finitely generated linear group either contains a non-abelian free subgroup or a solvable subgroup of finite index.     We consider here a somewhat stronger property.   If $G$ is a semisimple
algebraic group over a field $K$, then a \emph{strongly dense free subgroup} of $G$ over $K$ is a nonabelian free subgroup $\Gamma$ of $G(K)$ such
that every nonabelian subgroup $\Delta$ of $\Gamma$ is Zariski-dense over $K$ in $G$ (meaning that there is no proper closed $K$-subgroup $H$ over $G$ such that $\Delta\subset H(K)\subset G(K)$.) 

%See the  introduction of \cite{BGGT} for some remarks,  references and applications.

It was shown in \cite{BGGT, BGGT2}  that if  $G$ is a semisimple algebraic group over an 
algebraically closed field  $K$ of sufficiently large transcendence
degree over the prime field  (the degree depending upon $\dim G$), 
then $G$  always contains a  strongly dense  subgroup.   We gave applications to 
the Banach-Hausdorff-Tarski paradox and to generation properties for finite simple
groups of Lie type of bounded rank and also used this to study Cayley graph expanders
associated to finite simple groups of Lie type.   

In this note, we show the 
existence of such subgroups over much smaller fields.  
In particular, we prove:

\begin{theorem}     \label{char 0}  Let $K$ be a transcendental extension of $\Q$ and $G$ a split semisimple algebraic group over $K$.   
Then $G$ contains a strongly dense free subgroup over $K$. In fact, the set of pairs $(a,b) \in G(K)\times G(K)$ generating a strongly dense free subgroup is Zariski-dense in $G\times G$.
\end{theorem}

\begin{theorem} \label{char p}   Let $K$ be an extension of $\FF_p$ of transcendence degree $\ge 2$
and $G$ a split semisimple algebraic group over $K$.   
Then $G$ contains a strongly dense free subgroup over $K$. In fact, the set of pairs $(a,b) \in G(K)\times G(K)$ generating a strongly dense free subgroup is Zariski-dense in $G\times G$.
\end{theorem}

Letting $L = K(t)$ with $K$ a global field, we produce $d$-generated free subgroups $\Gamma$ of $G(A)$, where $A$ is a finitely generated $K$-subalgebra of $L$, so
that if we reduce modulo certain maximal ideals $\m_i$ of $A$, 
we get strongly dense free subgroups of 
the $K_i$-points of predetermined special subgroups $H_i$ of $G$, where $K_i = A/\m_i$.  These $H_i$  are 
images in $G$ of products of groups of the form $\SL_1(D_{i,j})$, where each $D_{i,j}$ is a central division algebra of prime degree over $K_i$.  These $\SL_1(D_{i,j})$ are special 
because all of their proper closed subgroups over $K_i$ are close to being abelian, and it is this property that enables us to prove strong density.
Then using some algebraic geometry and algebraic
group theory, we see that $\Gamma$, regarded as a subgroup of $G(L)$, must  be free and strongly dense in
$G$.    The strategy is not so dissimilar from that of  \cite{BGGT} where  degenerations of certain maximal rank subgroups of $G$ were studied.  However, a key
feature of the proof in \cite{BGGT} was the use of the  result of Borel \cite{borel-dom}  that word maps on algebraic groups are dominant
(see also \cite{larsen}).   This is not used  in our proof of the existence of free strongly dense subgroups, which uses the properties of special subgroups instead.
A major advantage of the new strategy is that it lets us work over fields of lower transcendence degree.

A key technical ingredient, which may be of independent interest, is Theorem~\ref{degeneration}, relating the Zariski closure of the generic image of a family of representations of a group $\Gamma$
to the Zariski closure of a member of the family.

We can also prove strong density results for certain non-free groups.

\begin{definition}
\label{strongly dense}
A nonabelian subgroup $\Gamma$ of a semisimple algebraic
group $G$ over $K$ is \emph{strongly dense} (over $K$) if every nonabelian subgroup of $\Gamma$ is Zariski-dense in $G$ over $K$.
\end{definition}

As an illustration of the definition, we prove that in a number of situations, $G(K)$ contains strongly dense surface groups.

\begin{theorem} \label{thm:surface}   Let $\Sigma_g, g \ge 2$ be the fundamental group of a compact orientable surface
of genus $g$ and $G$ a semisimple algebraic group over an algebraically closed field $K$.  Then there are strongly dense embeddings of $\Sigma_g$ in $G$ over $K$ if any of the following conditions holds:
\begin{enumerate}
\item[1.]$K$ has infinite transcendence degree over its prime field.
\item[2.]$K$ is transcendental over $\Q$ and $G$ is a classical group.
\item[3.]$K$ is of characteristic $p\ge 3$ and is of transcendence degree $\ge 10$ over $\F_p$, and $G$ is a special linear group.
\end{enumerate}
\end{theorem}

We also prove the existence of dense (or strongly dense) embeddings of other finitely generated groups $\Gamma$ in $G(K)$ under certain assumptions on $\Gamma$, $G$ and $K$. See Theorems  \ref{frf} and \ref{trans.deg.one}.

\section{Representation Varieties}

Throughout this section, we fix a linear algebraic group  $G = \Spec B$ 
over a field $K_0$.  
If $\Gamma$ is a $d$-generated group, we define
$\cF_\Gamma$ to be the functor from $K_0$-algebras $A$ to sets given by
$$A\mapsto \Hom(\Gamma,G(A)).$$

When $\Gamma$ is the free group $F_d$, the functor $\cF_\Gamma$ is represented by the affine $K_0$-scheme $G^d = \Spec B^{\otimes d}$, and evaluation at any element $\gamma$ in $\Gamma$ defines a morphism $e_\gamma$ of $K_0$-schemes $G^d\to G$. The fiber of $e_\gamma$  over the identity of $G$ defines a closed subscheme of $G^d$ associated to an ideal $I_\gamma$ in the coordinate ring $B^{\otimes d}$.  If $\Gamma = F_d/N_\Gamma$, we define
$$I_\Gamma = \sum_{\gamma\in N_\Gamma} I_\gamma.$$

We define $C_\Gamma = B^{\otimes d}/I_\Gamma$, so $\Spec C_\Gamma$ is a closed subscheme of $G^d$.
For every $K_0$-algebra $A$, an element of $\cF_\Gamma(A)$ corresponds to a homomorphism 
$\Hom(\Gamma,G(A))$, i.e., a homomorphism $F_d\to G(A)$ which restricts to the identity at each $\gamma\in N_\Gamma$.  Equivalently, it is an element of $G^d(A)$ which maps by $e_\gamma$ to the identity, i.e., an element of $\Hom_{K_0}(\Spec A, G)$ whose composition with $e_\gamma$ is the morphism $\Spec A\to G$ which factors through the identity map $\Spec K\to G$.  This means that $\Spec A\to G^d$ factors through $\Spec B^{\otimes d}/I_\gamma$ for all $\gamma\in I_\Gamma$, or, equivalently, a morphism $\Spec A\to \Spec C_\Gamma$.  Thus $\cF_\Gamma$ is representable by the affine scheme $\Spec C_\Gamma$, which we call the \emph{representation scheme}.

Let $A_\Gamma$ denote the quotient of $C_\Gamma$ by its nilradical.  If we restrict to $K_0$-algebras $A$ with radical $(0)$, we have
$$\Hom_{K_0}(C_\Gamma,A) = \Hom_{K_0}(A_\Gamma,A).$$
We define $\Spec A_\Gamma$ to be the \emph{representation variety} of $\Gamma$ (where $G$ and $K_0$ are understood).  In particular, the identity map on $A_\Gamma$ defines an element of $\cF_\Gamma(A_\Gamma)$, which we call the \emph{universal representation}.

This representation variety is irreducible if and only if $A_\Gamma$ is an integral domain.
If $K$ is any field extension of $A_\Gamma$, the universal representation defines a homomorphism $\Gamma\to G(K)$ which corresponds to a map from $\Spec K$ to the representation variety whose image is the generic point.  Various conditions on $\Gamma$-representations in $G$ determine closed subsets of the representation variety defined over $K_0$; unless such a condition holds for all representations, it does not hold for $\Gamma\to G(K)$ of the kind we are discussing.  Therefore, if $K$ is any algebraically closed field of transcendence degree $\ge \dim \Spec A_\Gamma$ over $K_0$, there exists a homomorphism $\Gamma\to G(K)$ satisfying no proper closed condition.  This can be useful, for example, in proving that there exists an injective homomorphism from $\Gamma$ to $G(K)$.

\section{Some Groups over Global Fields} \label{sec2} 
\label{sec:reduce} 

We first point out the following easy result (see also \cite[Prop. 4.1]{GG}).  

\begin{lemma}
\label{normalizer}
Let $K$ be any field, $p$ a prime,  $D$ a central division algebra of degree $p$ over $K$, and 
$G = \SL_1(D)$ the inner form of $\SL_n$ over $K$ associated to $D$.  Then every proper $K$-subgroup of positive 
dimension of $G$ is contained in the normalizer of a maximal torus of $G$.
\end{lemma}

\begin{proof}
Let $H$ denote a proper connected subgroup of $G$.
As $G(K) \subset D^\times$, it contains no non-trivial unipotent elements, so the same must be true for $H$. 
 In particular, $H$ is reductive. 
If $S$ denotes the derived group of $H$, it is semisimple of dimension less than $p^2-1$ and therefore of rank $r<p-1$.
Let $T$ denote a maximal torus of $S$ defined over $K$.  Let $K^s$ denote a separable closure of $K$, and $X\cong \Z^r$ the character group of $T$.  Then $T$ determines a continuous homomorphism $\Gal(K^s/K)\to \Aut(X)\cong \GL_r(\Z)$ whose image $I$  is finite.  There are no elements of order $p$ in $I$ since the minimal polynomial of a primitive $p$th root of unity has degree $p-1>r$.  Thus $T$ splits over a Galois extension $L/K$ such that $\Gal(L/K)\cong I$ is of prime to $p$ order.   Therefore the class of $D$ does not lie in the kernel of 
$\Br(K)\to \Br(L)$, and $D_L := D\otimes_K L$ is a central division algebra over $K$.

Thus $\SL_1(D_L)$ is an algebraic group over $L$ containing a non-trivial split torus $T_L$, i.e., an isotropic semisimple group.  It follows that $\SL_1(D_L)$ contains non-trivial unipotents \cite[3.4 (iii)]{Springer}, which is absurd since $D_L$ has no non-trivial nilpotents.

\end{proof}

In fact over a global field, groups of type $A_m$ with $m+1$ prime are the only simple algebraic groups 
that have a form with the above property  (called almost abelian).  See \cite{GG} for more about this.

\begin{lemma}  
\label{A1}   
Let $K$ be a global field and $D_1,\ldots, D_m$ $K$-central division algebras of prime degrees $p_1,\ldots,p_m$ respectively such that $\SL_1(D_1),\ldots,\SL_1(D_m)$ are pairwise non-isomorphic.
Then
$$G(K)=G_1(K)\times \cdots \times G_m(K) = \SL_1(D_1)\times \cdots \times \SL_1(D_m)$$ 
contains a strongly dense free subgroup.
\end{lemma}

\begin{proof}
By the Tits alternative, there exist elements $x_i, y_i\in \SL_1(D_i)$ which generate a free subgroup $F_2$.  Let $\Gamma = \langle (x_1,\ldots,x_m),(y_1,\ldots,y_m)\rangle$.  Since $x_1$ and $y_1$ satisfy no word relation, the same is true for
$(x_1,\ldots,x_m)$ and $(y_1,\ldots,y_m)$, so $\Gamma$ is isomorphic to $F_2$.  Let $\Delta\subset \Gamma$ be a non-abelian subgroup.  To show that $\Delta$ is Zariski-dense in $G$, by Goursat's lemma and the fact that the $G_i$ are simple as algebraic groups and pairwise non-isomorphic, it suffices to prove that the projection of $\Delta$ on each factor $G_i$ is Zariski-dense.  

Let $H_i$ denote the Zariski closure of the image of $\Delta$ in $G_i$.  Every proper subgroup of $G_i$ is contained in the normalizer of a maximal torus, so we may assume $H_i$ normalizes $T_i$, but as $\Delta$ is non-abelian, it is a free group on $\ge 2$ generators and therefore not contained in the normalizer of a torus.

\end{proof}

\begin{corollary}
\label{special}
With $K$ and $p_i$ as above, there exists a finite separable extension $L/K$ such that $\SL_{p_1}\times\cdots\times  \SL_{p_m}$ contains a strongly dense free subgroup over $L$.
\end{corollary}

\begin{proof}
Let $L$ be a finite separable extension of $K'$ which splits all the $\SL_1(D_i)$.
\end{proof}

 \section{Degenerations}

In this section, we introduce a variant of the notion of \emph{degeneration} 
introduced in \cite{BGGT}.
Let $K_0$ be a field, $A$ a $K_0$ algebra, and $n$ a positive integer.  Let $P_1$ and $P_2$ denote prime ideals of $A$ such that $P_1\subset P_2$, and let $K_i$ denote the field of fractions of $A/P_i$.
Let $\Gamma$ be a finitely generated subgroup of $\GL_n(A)$ and let $G_1$ and $G_2$ denote the Zariski closures over $K_1$ and $K_2$ respectively, of the images of $\Gamma$ in $\GL_n(K_1)$ and $\GL_n(K_2)$
respectively.

\begin{definition}
In this situation, we say \emph{$G_2$ is a degeneration of $G_1$} or \emph{$G_1$ can degenerate to $G_2$}.
\end{definition}

Note that if $G_1$ can degenerate to $G_2$, then $G_1^\circ$ can degenerate to $G_2^\circ$.
Indeed $\Gamma$ maps to $G_1(K_1)\times G_2(K_2)$, and the inverse image of $G_1^\circ(K_1)\times G_2^\circ(K_2)$ is of finite index in $\Gamma$.  For $i=1,2$, replacing $\Gamma$ by this finitely generated subgroup replaces $G_i$ by a finite index subgroup of $G_i^\circ$ and therefore with $G_i^\circ$ itself.

\begin{proposition}
Let $P_1\subset P_2\subset A$, $K_i$, $\Gamma\subset \GL_n(A)$ a finitely generated group, and $G_i$ be as above.  Let $X$ be a projective scheme over $A$ endowed with an action of $(\GL_n)_A$ defined over $A$.  
If $(G_1)_{\overline{K_1}}$ has a fixed point on $X_{\overline{K_1}}$, then $(G_2)_{\overline{K_2}}$ has a fixed point on $X_{\overline{K_2}}$.
\end{proposition}

\begin{proof}
Let $\Gamma = \langle \gamma_1,\ldots,\gamma_d\rangle$.  The embedding of $\Gamma$ in 
$\GL_n(A)$ defines a section of $(\GL_n)_A^d\to \Spec A$.  The image is therefore closed in
$(\GL_n)_A^d$.  Let $Y$ denote the inverse image of $X$, diagonally embedded in $X^{d+1}$,
under the morphism $(\GL_n)_A^d\times X\to X^{d+1}$ given by 
$$(g_1,\ldots,g_d,x)\mapsto (g_1.x,\ldots,g_d.x),$$
and let $Z$ be the image of $Y$ under the projection map from $(\GL_n)_A^d\times X$
to $(\GL_n)_A^d$.  As $X$ is complete, the projection map is proper and therefore closed,
so $Z$ is a closed subset of $\Spec A$.  As $(G_1)_{\overline{K_1}}$ has a fixed point on $X_{\overline{K_1}}$, $Z$ contains $P_1$.  As $P_2$ lies in the closure of $P_1$, it contains $P_2$ as well, so choosing a $\overline{K_2}$-point of $Y$ which lies over $P_2$, $(G_2)_{\overline{K_2}}$ has a fixed point on $X_{\overline{K_2}}$.\end{proof}

This has the following consequence:

\begin{corollary}  \label{cor:pv}
With notations as above, if $G_1$ fixes a $k$-dimensional subspace of $\overline{K_1}^n$, then $G_2$ fixes a $k$-dimensional subspace of $\overline{K_2}^n$.
\end{corollary}

\begin{proof}  This follows from the previous result using the fact that the Grassmannian variety of $d$-dimensional
subspaces is a complete variety.
\end{proof} 

As $K_1$ and $K_2$ are both extensions of $K_0$, we have $\cha K_1 = \cha K_2$, so we can embed both fields as subfields of a common algebraically closed field, and it then makes sense to ask whether $G_2$ is conjugate (in $\GL_n$) to a subgroup of $G_1$.  In general, this is not the case.  For instance, if $A = \C[t]$, $P_1 = (0)$, $P_2 = (t)$, and $\Gamma$ is the (infinite cyclic) subgroup of $\GL_2(A)$ generated by
$\begin{pmatrix}1&1\\ 0&1+t\end{pmatrix}$, then
$G_1\cong  \G_m$ and $G_2\cong \G_a$.  

For any linear algebraic group $G$, we define $G^{\ss}$ to be the (connected and reductive) quotient of the identity component $G^\circ$
by its unipotent radical $\Rad_u(G^\circ)$.
Every homomorphism of algebraic groups $\phi\colon G\to H$ over an algebraically closed field $K$ determines a homomorphism $\psi\colon G^\circ\to H^{\ss}$.  As $U:=\psi(\Rad_u G^\circ)\subset H^{\ss}$
is connected and unipotent, it is contained in the unipotent radical of a canonically defined parabolic subgroup $P\subset H^{\ss}$ \cite[\S30]{humphreys}, so $\psi(G^\circ)\subset P$.
Letting $P=MN$ denote a Levi decomposition, $U\subset N$, so $\psi$ induces a well-defined homomorphism $G^{\ss}\to M$.  As Levi decomposition is unique up to conjugation \cite[Theorem 30.2]{humphreys},
composition of $\psi$ with $M\hookrightarrow H^{\ss}$ gives a morphism $\phi^{\ss}\colon G^{\ss}\to H^{\ss}$, which is well-defined up to conjugation.  If $G^\circ \cap \ker \phi$ is connected and unipotent then the same is true of $\ker\psi$, so $\phi^{\ss}$ is injective.

\begin{lemma}
\label{compatible}
Let $G$ and $H$ be linear algebraic groups over an algebraically closed field $K$ and $\phi\colon G\to H$ and $\rho\colon H\to \GL_n$ homomorphisms defined over $K$.  Then the representations of $G^{\ss}$ defined by $\rho^{\ss}\circ \phi^{\ss}$ and $(\rho\circ\phi)^{\ss}$ have isomorphic semisimplifications.
\end{lemma}

Note that the condition on semisimplifications just means that every irreducible representation of $G^{\ss}$
has the same multiplicity as a constituent of the two given representations.
\\

\begin{proof}
It suffices to prove that the characteristic polynomials of $\rho^{\ss}\circ \phi^{\ss}(g^{\ss})$
and $(\rho\circ\phi)^{\ss}(g^{\ss})$ are the same for all $g^{\ss}\in G^{\ss}(K)$.  
Let $g\in G^\circ(K)$ map to $g^{\ss}$.
Writing $\psi(g) = mn\in M(K)N(K)=P(K)\subset H^{\ss}(K)$, we have $\phi^{\ss}(g^{\ss}) = m$.

We denote by $\tilde N$ and $\tilde P$ the inverse images of $N$ and $P$ respectively in $H^\circ$.
Let $h\in H^\circ(K)$ map to $m$, so $\phi(g) = hu$ for some $u\in \tilde N(K)$.  We claim that $\rho(h)$ and $\rho(hu)$ have the same characteristic polynomial.

Indeed, $\rho(\tilde N)$ is a connected unipotent subgroup of $\GL_n$, so there exists a parabolic subgroup $Q$ of $\GL_n$ such that $\rho(\tilde N)\subset \Rad_u(Q)$ and $\rho(\tilde P)\subset  Q$.
Thus, $\rho(u)\in\Rad_u(Q)(K)$ and $\rho(h)\in Q(K)$.  So, indeed $\rho(h)$ and $\rho(h)\rho(u)$
have the same characteristic polynomial.
\end{proof}

If $\rho_1\colon G_1\to \GL_n$ and $\rho_2\colon G_2\to \GL_n$ are defined over $K$,
for a morphism $\phi\colon G_1^{\ss}\to G_2^{\ss}$ to be \emph{compatible with $\rho_i$} means
$\rho_2^{\ss}\circ \phi$ and $\rho_1^{\ss}$ define representations of $\G_1^{\ss}$ which have isomorphic semisimplifications.

\begin{lemma}
\label{semisimplify}
Let $\Gamma$ be a group and $f_1,f_2\colon \Gamma\to \GL_n(K)$ homomorphisms defining representations whose semisimplifications are isomorphic.
Let $G_1$ and $G_2$ denote the
Zariski closure of $f_1(\Gamma)$ and $f_2(\Gamma)$ respectively and $\rho_1$, $\rho_2$ the inclusion morphisms from $G_1$ and $G_2$ to $\GL_n$.  Then there exists an isomorphism between $G^{\ss}_1$ and $G^{\ss}_2$ compatible with $\rho_i^{\ss}$.

\end{lemma}

\begin{proof}
Passing to a finite index normal subgroup of $\Gamma$, we may assume that $G_1$ and $G_2$ are connected.  As the maps $\rho_i$ are injective, the same is true for $\rho_i^{\ss}$.

For all $g\in \Gamma$ let $g_i$ denote the image of $f_i(g)$ in $G_i^{\ss}(K)$.  Then
$\rho_i^{\ss}(g_i)$ has the same characteristic polynomial as $f_i(g_i)$, which is therefore the same for $i=1$ and $i=2$.  Let $G_{12}$ denote the Zariski closure of $\Gamma$ under the diagonal map to $G_1^{\ss}(K)\times G_2^{\ss}(K)$.  There are two morphisms from $G_{12}$ to characteristic polynomials, one via $\rho_1^{\ss}$ and one via $\rho_2^{\ss}$, and they must coincide.

We claim that $G_{12}$ is the diagonal of an isomorphism between $G_1^{\ss}$ and $G_2^{\ss}$.
Otherwise, there exists a non-trivial normal subgroup $H$ of, say, $G_1^{\ss}$ such that $H\times \{1\}\subset G_{12}$, and $H$ maps under $\rho_1^{\ss}$ to the closed subvariety of unipotent elements.  This is impossible, so $H$ gives an isomorphism between $G_1^{\ss}$ and $G_2^{\ss}$ compatible with the $\rho_i^{\ss}$.

\end{proof}

\begin{theorem}
\label{quasi-inj}
Let $K$ be a field, 
$$f\colon \Gamma\to \GL_n(K[[t]])\subset \GL_n(K((t)))$$
a representation of a group $\Gamma$, and $\bar f\colon \Gamma\to \GL_n(K)$  the  reduction of $f$ (mod $t$).
Let $G$ (resp. $\bar G$) denote the Zariski closure in $\GL_n$ over $K((t))$ (resp. $K$) of $f(\Gamma)$ (resp. $\bar f(\Gamma)$), and let $\rho$ (resp. $\bar\rho$) denote the inclusion $G\to\GL_n$ as a morphism of algebraic groups over $K((t))$ (resp. $K$).
Then there exists a finite extension $L$ of $K((t))$ and
an injective homomorphism $\bar G^{\ss}\times_K L\to G^{\ss}\times_{K((t))}L$ compatible with $\rho^{\ss}\times_{K((t))}L$ and $\bar\rho^{\ss}\times_K L$.

\end{theorem}

\begin{proof}
Replacing $\Gamma$ with a finite index subgroup, we may assume without loss of generality that $G$ and $\bar G$ are connected.

Let $V_1\subset V_2\subset\cdots\subset V_k = K((t))^n$ denote a flag of $K((t))$-spaces preserved by $f(\Gamma)$ and such that $f(\Gamma)$ acts irreducibly on each $V_{j+1}/V_j$.  Let $\Lambda_j = V_j\cap K[[t]]^n$, so that $f(\Gamma)$ preserves each $\Lambda_j$ and $\Lambda_{j+1}/\Lambda_j$ is a free $K[[t]]$-module of rank $\dim V_{j+1}/V_j$.  We may therefore fix a $K[[t]]$-free complement $M_{j+1}$
to $\Lambda_j$ in $\Lambda_{j+1}$.  We have a surjective homomorphism from the stabilizer of the flag $V$
in $\GL_n(K[[t]])$ to $\prod_{j=1}^k \Aut_{K[[t]]} M_j$, and a section of this latter group whose image consists of $K[[t]]$-linear maps preserving the direct sum decomposition given by the $M_j$.  Replacing $f$ by the composition with these two homomorphisms, we obtain a semisimplification of $f$ which still lands in $\GL_n(K[[t]])$ and such that $\bar f(g)$ has the same characteristic polynomial as $\bar f(g)$ for all $g\in \Gamma$.  Therefore, without loss of generality, we may assume $f$ is semisimple.

By \cite[Chap.~2, \S2]{SerreLF}, every finite extension of $K((t))$ is complete with respect to the unique extension of the $t$-adic valuation on $K((t))$.  Each such extension is therefore of the form $K'((t'))$, where
$K\subseteq K'$, $K[[t]]\subseteq K'[[t']]$, and $t$ is a non-unit in $K'[[t']]$.
As every semisimple group splits over a finite extension, we may assume that $G$ is split over $K((t))$.

\def\B{\mathcal{B}}
\def\Int{\mathrm{Int}}

Next, we claim that $f(\Gamma)$ is a bounded subgroup of $G(K((t)))$ in the sense of Bruhat-Tits \cite[4.2.19]{BT2}.  Indeed, we can fix a finite set of generators of the coordinate ring of $G$.
As $\rho$ is a closed immersion, each generator lifts to an element of the coordinate ring of $\GL_n$.
As $f(\Gamma)$ is bounded in the $\GL_n$ sense, it is bounded in the $G$ sense as well.
By  \cite[3.3.1]{BT1}, it stabilizes the centroid of some facet of the Bruhat-Tits building $\B(G/K((t)))$.

By \cite[2.4]{Larsen95}, replacing $K((t))$ by a finite extension, we may assume that $f(\Gamma)$ stabilizes a hyperspecial vertex of $\B(G/K((t)))$.
By \cite[4.6.22]{BT2}, there exists a split semisimple group $G_0$ over $K$ and an isomorphism $\iota$ from $G$ to $G_0\times_K K((t))$ such that $\iota(f(\Gamma))\subset G_0(K[[t]])$.  By \cite{Tits71}
there exists a homomorphism  $\rho_0\colon G_0\to \GL_n$  of algebraic groups over $K$ 
whose extension of scalars to $K((t))$,  via $\iota$, defines the same representation as $\rho$. 
Explicitly, this means there exists $g\in \GL_n(K((t)))$ such that 
$$(\rho_0\times_K {K((t))})\circ \iota = \Int(g)\circ \rho$$
as $K((t))$-morphisms $G \to \GL_n$.  Thus,
$$(\rho_0\times_K {K((t))})\circ\iota(f(\Gamma)) = g^{-1} f(\Gamma) g.$$

By the Brauer-Nesbitt theorem, the (mod $t$) reductions of $f(\Gamma)$ and $g^{-1}f(\Gamma) g$ have the same semisimplification.
By Lemma~\ref{semisimplify}, without loss of generality we may assume $g=1$.

Replacing $f$ by $\iota\circ f$, we have therefore reduced to the situation that $G$ is a  semisimple group over $K$ embedded in $\GL_n$ by the $K$-homomorphism $\rho_0$, and $f(\Gamma)\subset G_0(K((t)))$ is a Zariski dense subgroup which lies in $G_0(K[[t]])$.  Reduction (mod $t$) commutes with $\rho_0$.  Thus, the Zariski-closure $\bar G$ of $\bar f(\Gamma)$ in $\GL_n$ over $K$ is contained in $\rho_0(G)$.  The Zariski closure $G$ of $f(\Gamma)$ in $\GL_n$ over $K((t))$ is $\rho_0(G)\times_K K((t))$.  

We conclude that there is an injective homomorphism of algebraic groups over $K((t))$, $\bar G\times_K K((t))\to G$.  By Proposition~\ref{compatible}, the injective homomorphism of algebraic groups
$\bar G^{\ss}\times_K K((t)) \to G^{\ss} = G$ is compatible with $\bar\rho^{\ss}$ and $\rho$ up to semisimplification.
The theorem follows.

\end{proof}

\begin{theorem}
\label{degeneration}
If $G_2$ is a degeneration of $G_1$ and $\rho_i$ denotes the inclusion of $G_i$ in $\GL_n$,
then there exists a field $L$ such that after extending scalars to $L$,
there exists an injective homomorphism $G_2^{\ss}\to G_1^{\ss}$ compatible with $\rho_i^{\ss}$.
\end{theorem}
 
\begin{proof}
As there exists a finite chain of prime ideals maximal among all chains connecting $P_1$ and $P_2$,
without loss of generality, we may and do assume that there is no prime ideal intermediate between $P_1$ and $P_2$.  Replacing $A$ by the localization of $A/P_1$ at $P_2/P_1$
and $\Gamma$ by its image in $\GL_n((A/P_1)_{P_2/P_1})$,
we may assume without loss of generality that $A$ is a $1$-dimensional local domain, $P_1=0$, and $P_2$ is a maximal ideal.

Let $\tilde A$ denote the normalization of $A$.  The morphism $\Spec \tilde A\to \Spec A$ is birational and finite \cite[Tag 0BXR]{Stacks}, so its image is closed and contains the generic point of $\Spec A$.
Therefore, there exists a prime ideal $\tilde P_2$ of $\tilde A$ lying over $P_2$.  Replacing $A$ by $\tilde A$ and $P_2$ by $\tilde P_2$, we may assume that $A$ is a DVR.
Replacing $A$ by its completion at $P_2$, we may assume that it is a complete
DVR, $P_1$ is the zero-ideal, and $P_2$ is the maximal ideal.  The fields $K_1$ and $K_2$ are respectively the fraction field and the residue field of $A$.  By Cohen's classification of complete equicharacteristic regular local rings \cite[Theorem 15]{Cohen},
we have isomorphisms $A\cong K_2[[t]]$, and $K_1 \cong K_2((t))$.  

By the previous theorem, there exists $L$ such that $G_2^{\ss}\times_{K_2} L$ is isomorphic to a
closed subgroup of $G_1^{\ss}\times_{K_1} L$, compatibly with $G_i^{\ss}\to \GL_n$, so the theorem follows. 
\end{proof}

If $H_1$ and $H_2$ are closed subgroups of $\GL_n$ over $K$ we write $H_1\prec_{\GL_n} H_2$ if and only if there exists an extension $L/K$ and an injective homomorphism $H_1^{\ss}\times_K L\hookrightarrow H_2^{\ss}\times_K L$
compatible with the given $n$-dimensional representations of $H_1\times_K L$ and $H_2\times_K L$.
If $H_2$ can degenerate to $H_1$ as subgroups of $\GL_n$, then $H_1\prec H_2$.  If $H_2$ and $H_1$ are subgroups of a semisimple group $G$, we write $H_1\prec H_2$ if and only if $H_1\prec_{\GL_n}H_2$ for all faithful representations $G\hookrightarrow \GL_n$.  If $K$ is algebraically closed, this implies that $H_1^{\ss}$ is isomorphic to a closed subgroup of $H_2^{\ss}$ over $K$ itself.

Our strategy for constructing strongly dense free subgroups of $G$ over a transcendental extension $K$ of a global field $K_0$ is to find \emph{special} semisimple subgroups
$H_1,\ldots,H_r$ of $G$ and homomorphisms $\rho_i\colon \Gamma\to H_i(K_0)$ which are strongly dense thanks 
to Corollary~\ref{special}.  We then construct a curve of homomorphisms $\rho\colon \Gamma\to G(A)$ which specializes
to $\rho_1,\ldots,\rho_r$ at different points of the curve $\Spec A$.  The Zariski closure $H$ of $\rho(\Delta)$ for any non-abelian free subgroup of $\Gamma$
then satisfies $H_i\prec H$ for all $i$, thanks to Theorem~\ref{degeneration}.  The goal of the next section is to find choices of $H_i$ for which these conditions imply $H=G$.

\section{Special Subgroups}

In this section, we gather some results about subgroups of simple algebraic groups.  Throughout this section, we always assume $K$ is algebraically closed and let $p$ be the characteristic of $K$ when it is non-zero.

\begin{definition}
We say a closed subgroup $H$ of a linear algebraic group $G$ is \emph{special} if it is the image of a homomorphism $\SL_{p_1}\times \cdots \times \SL_{p_m}\to G$ for some sequence of primes $p_1,\ldots,p_m$.  We say it is \emph{very special} if we can take $m=1$ and $p_1=2$.
\end{definition}

Note that if $K$ contains a global field $K_0$, every $K$-group of the form $\SL_{p_1}\times \cdots \times \SL_{p_m}$ over $K$ is obtained from a product of almost abelian simple algebraic groups of the form $\SL_1(D_1)\times \cdots\times \SL_1(D_m)$ for some $K_0$-central division algebras $D_1,\ldots,D_m$.  Moreover we may assume that each $D_i$ ramifies over some place that no other $D_j$ ramifies over.

\begin{definition}
A collection $\{H_1,\ldots,H_k\}$ of special subgroups of $G$ is \emph{generating} if the only subgroup $H$ of $G$ satisfying $H_i\prec H$ for all $i$ is $G$ itself.
\end{definition}

Note that if $G$ has a generating collection of special (resp. very special) subgroups, then the same is true for all groups isogenous to $G$.  Indeed, it is true for the universal covering group $\tilde G$ of $G$ since groups of the form
$\SL_{p_1}\times \cdots \times\SL_{p_m}$ are simply connected, so every homomorphism $\SL_{p_1}\times \cdots \times\SL_{p_m}\to G$ lifts to $\SL_{p_1}\times \cdots \times \SL_{p_m}\to \tilde G$.  It is clearly true for any quotient of a group for which it is true.

\begin{lemma}
\label{partition}
Let $k$ and $n$ be positive integers with $k<n$.  Then there exists a partition $\pi$ of $n$ such that no part has a prime factor greater than $3$, and $k$ is not a sum of any subset of the parts of $\pi$.
\end{lemma}

\begin{proof}
Let $n$ be the smallest integer for which the statement fails.
As $n$ cannot be of the form $2^a 3^b$, we may assume $n\ge 5$.
Let $m$ denote any integer of the form $2^a 3^b$ which lies in $(n/2,n)$.  
Thus, $m\le n$, and $n-m<m$.
If $k<n-m$ or $k>n-m$ but $k\neq m$, then by assumption, there exists a partition $\pi'$ of $n-m$ such that
no part of $\pi'$ has a prime factor greater than $3$, and $k$ is not a sum of parts of $\pi'$.
If $\pi$ is the partition of $n$ obtained by adding the part $m$ to $\pi'$, then no part of $\pi$
has a prime factor $>3$, and $k$ is not a sum of parts of $\pi$, contrary to assumption.

Therefore, it suffices to prove that there are at least two different values of $m\in (n/2,n)$, $m_1$ and $m_2$,
neither of which has a prime factor greater than $3$.  We can take $m_1$ to be the smallest power of $2$ greater than $n/2$ and $m_2$ to be $3$ times the smallest power of $2$ greater than $n/6$.
\end{proof}
 
 \begin{proposition}
 \label{A char not 2}
 Let $n$ be a positive integer and $K$ a field not of characteristic $2$.
 For any positive integer $k<n$, there exists a very special subgroup $H$ of $\SL_n$ defined over $K$
 such that the restriction of the natural representation of $\SL_n$ to $H$ has no $k$-dimensional subrepresentation.
 \end{proposition}
 
 \begin{proof}
 If $K$ is of characteristic $0$, we may take $H$ to be the image of $\SL_2$ under the symmetric $(n-1)$st power map.  We therefore assume that $p>2$.
 
By Lemma~\ref{partition}, there exists a partition $n = \pi_1+\pi_2+\cdots+\pi_r$, where each $\pi_i$ is of the form $2^{a_i} 3^{b_i}$ for non-negative integers $a_i$ and $b_i$ and $k$ is not a sum of any subsequence of terms in the sequence $\pi_1,\ldots,\pi_r$.
Let $V_2$ and $V_3$ denote, respectively, the representation space of the natural representation of $\SL_2$ and its symmetric square, which is irreducible since $p\ge 3$.  Let
$$W_i = V_2 \otimes V_2^{(p)} \otimes \cdots \otimes V_2^{(p^{a_i-1})} 
\otimes V_3 \otimes V_3^{(p)} \otimes \cdots \otimes V_3^{(p^{b_i-1})},$$
where $\SL_2(K)\to \GL(V^{(q)})$ denotes the composition of the $q$-Frobenius on $\SL_2$
with the representation $\SL_2\to \GL(V)$.  Thus $W_i$ is an irreducible representation of $\SL_2$ of degree $\pi_i = 2^{a_i}3^{b_i}$.  The direct sum of the $W_i$ is therefore a semisimple determinant $1$ representation of $\SL_2$ with no $k$-dimensional subrepresentation.
 \end{proof}

 \begin{lemma}
 \label{A irred}
 Let $n$ be a positive integer and $K$ a field.
There exists a special subgroup $H$ of $\SL_n$ 
such that the restriction of the natural representation of $\SL_n$ to $H$ is irreducible.
 \end{lemma}
 
\begin{proof}
Let $n=p_1^{e_1}\cdots p_l^{e_l}$, let 
$$H=\prod_{i=1}^l \SL_{p_i}^{e_i},$$
and embed $H$ in $\SL_n$ via the external tensor product of the natural representation of the $\SL_{p_i}$ factors.
\end{proof}

We shall also need a special case of a result of McLaughlin \cite{McL}. 

\begin{proposition}
\label{5.6}
Suppose that  $n\ge 2$, and $H$ is a semisimple subgroup of $G=\GL_n$ acting irreducibly on the natural representation $V$ of $G$.  
If $H$ contains a subgroup $H_1\cong \SL_2$ such that
the restriction of $V$ to $H_1$ is $V_0^{n-2}\oplus V_1$,
then $H=\Sp_n$ or $H=\SL_n$. 
\end{proposition}

\begin{proposition}
\label{C irred}
Let $K$ be a field and $G = \Sp_{2r}$ a symplectic group over $K$.  If $\cha K=3$, we further assume $r\ge 3$.  Then there exist special subgroups $H_1$ and $H_2$ of $G$ such that for all $k$ in $[1,2r-1]$, there exists $i$ such that
the restriction of the natural representation of $G$ to $H_i$ has no $k$-dimensional $H_i$-invariant subspace.   If $\cha K = 0$, $H_1$ alone suffices, and we may take it to be very special.
\end{proposition}

\begin{proof}
Let $V_i$ denote the symmetric $i$th power representation of $\SL_2$.
If $K$ has characteristic $0$ or characteristic $\ge 2r$, then we do not need $H_2$; we may take $H_1\subset \GL_{2r}$ to be the image of $\SL_2$ under the representation $V_{2r-1}$.  As $V_{2r-1}$ is irreducible and symplectic, $H_1$ may be taken to be 
a subgroup of $\Sp_{2r}$.  We therefore assume $\cha K = p$,
where $2\le p\le 2r-1$.  

If $r= p_1^{e_1}\cdots p_l^{e_l}$, then $\prod_i \SL_{p_i}^{e_i}$ embeds in $\SL_r$ by the tensor product of natural representations, and thence into $\Sp_{2r}$.  Defining $H_1$ to be the image of this representation, the restriction of the natural representation of $G$ to $H_1$
is the direct sum of two  irreducible factors of dimension $r$.  We may therefore assume, henceforth, that $k=r$.

If $\cha K=2$, then tensor products of distinct Frobenius twists of $V_1$ give irreducible self-dual representations of $\SL_2$ of every degree in the set $\{1,2,4,8,\ldots\}$ and, in particular, one whose degree lies in $[r+1,2r]$.
Adding a trivial representation of suitable degree, we obtain a self-dual representation $V$ of degree $2r$ which has no invariant $k$-dimensional subspace.  Since $\cha K=2$, the image $H_1$ of $\SL_2$ in this representation can be taken to be in $\Sp_{2n}$.  We may therefore assume $p\ge 3$.  Since we are excluding the case $(p,r) = (3,2)$, we may assume $r\ge 3$.

We divide into cases.

\vskip 3pt
\noindent \textbf{Case $r\neq 8$.}
We know $V_1$ is  irreducible and symplectic of degree 2, $V_2$ is  irreducible and orthogonal of degree $3$, and the tensor product of two distinct Frobenius-twists of $V_1$ is  irreducible and orthogonal of degree $4$.
By tensoring suitable Frobenius twists of these three representations, we can therefore find an irreducible symplectic representation of $\SL_2$ of any degree of the form $2^{2e+1} 3^f$.  Our assumptions on $r$ guarantee there exists a number of this form in
the interval $[r+1,2r]$.   Therefore, there exists a symplectic,  irreducible representation of $\SL_2$ whose degree lies in this interval.  The direct sum of this irreducible representation with a trivial representation of suitable degree
is then a symplectic representation of $\SL_2$ of degree $2r$, and defining $H_2$ to be the image of $\SL_2$ in this representation, the proposition follows in this case.
\vskip 3pt
\noindent \textbf{Case $r = 8$, $p\ge 5$.}
Let $W$ denote the adjoint representation of $\SL_3$, which is orthogonal and irreducible.  Then $V_1\boxtimes W$ is a $16$-dimensional irreducible symplectic representation of $\SL_2\times \SL_3$, so we may take $H_2$ to be the image of this representation.  (In fact, in this case, we do not need $H_1$.)
\vskip 3pt
\noindent \textbf{Case $r = 8$, $p=3$.}  Let $W$ denote the semisimplification of the adjoint representation of $\SL_3$, which is orthogonal and the direct sum of a trivial $1$-dimensional representation and an  irreducible $7$-dimensional representation.  Then $V_1\boxtimes W$ is a $16$-dimensional symplectic representation of $\SL_2\times \SL_3$ which decomposes into irreducible factors of degree $2$ and $14$, and we let $H_2$ denote the image of this representation.

\end{proof}

\begin{proposition}
\label{B irred}
Let $K$ be a field not of characteristic $2$ and $G = \SO_{2r+1}$, $r\ge 3$.  Then there exist special subgroups $H_1,H_2,H_3$ of $G$ such that for all $k$ in $[1,2r]$, there exists $i$ such that
the restriction of the natural representation of $G$ to $H_i$ has no $k$-dimensional $H_i$-invariant subspace.  If $\cha K = 0$, $H_1$ alone suffices, and we may take it to be very special.
\end{proposition}

\begin{proof}
If $K$ has characteristic $0$ or characteristic $\ge 2r+1$, then we do not need $H_2$ or $H_3$.  We may take $H_1\subset \GL_{2r+1}$ to be the image of $\SL_2$ in the representation $V_{2r}$.  As $V_{2r}$ is irreducible and orthogonal, $H_1$ may be taken to be 
a subgroup of $\SO_{2r+1}$.  We therefore assume $\cha K = p$,
where $3\le p\le 2r$.  

If $r= p_1^{e_1}\cdots p_l^{e_l}$, then $\prod_i \SL_{p_i}^{e_i}$ embeds in $\SL_r$ by the tensor product of natural representations, and thence into $\SO_{2r}\subset \SO_{2r+1}$.  Defining $H_1$ to be the image of this representation in $\SO_{2r+1}$, the restriction of the natural representation of $G$ to $H_1$
is the direct sum of a $1$-dimensional trivial factor and two  irreducible factors of dimension $r$.  We may therefore assume, henceforth, that $k\in\{1,r,r+1,2r\}$.
There is a special subgroup of $\SO_{2r+1}$ of the form $H_2=\SO_3^e \times \SO_4^f$, where $3e+4f=2r+1$, and no composition factor of the restriction of the natural representation of $G$  to $H_2$ has dimension $1$, so we may assume that $k\in \{r,r+1\}$.
It therefore suffices to find $H_3$ such that the restriction of the natural representation to $H_3$ has an irreducible factor of degree $\ge r+2$.

If $r=3$ and $p=3$, we let $H_3$ be the image of $\SL_3$ under the (irreducible, orthogonal) quotient of the adjoint representation by its $1$-dimensional trivial subrepresentation.
If $r=3$ and $p\ge 5$, we let $H_3$ denote the image of $\SL_2$ under $V_4$, which is, again,  irreducible and orthogonal.

If $r\ge 4$, then there exists an integer of the form $3^e 4^f$ in the interval $[r+2,2r+1]$.  Indeed, we take the smallest integer in the set
$$\{9\cdot 4^f\mid f\in\N\}\cup \{12\cdot 4^f\mid f\in\N\}\cup\{16\cdot 4^f\mid f\in\N\}\cup \{27\cdot 4^f\mid f\in\N\}$$
which exceeds $r+1$.  We can therefore find an  irreducible, orthogonal representation of $\SL_2$ by tensoring together $e$ suitable Frobenius twists of $V_2$ and $f$ suitable twists of $V_1$.  Adding a trivial factor of suitable dimension, we obtain an orthogonal representation which has no subrepresentation of degree $r$ or $r+1$; we take $H_3$ to be the image of $\SL_2$ under such a representation.

\end{proof}

\begin{proposition}
\label{D irred}
Let $K$ be a field and $G = \SO_{2r}$, $r\ge 4$.  Then there exist special subgroups $H_1$ and $H_2$ of $G$ such that for all $k$ in $[1,2r-1]$, there exists $i$ such that
the restriction of the natural representation of $G$ to $H_i$ has no $k$-dimensional $H_i$-invariant subspace.    If $\cha K = 0$, we may take $H_1$ and $H_2$ to be very special.
\end{proposition}

\begin{proof}
If $\cha K = 0$ or if $\cha K = p$ is at least $2r-1$, then $V_{2r-2}$ is  irreducible and orthogonal, and it maps $\SL_2$ to $\SO_{2r-1}$ and therefore to $\SO_{2r}$.  The restriction of the natural representation of $\SO_{2r}$ to the image $H_1$ of this homomorphism decomposes as the sum of a $1$-dimensional trivial representation and an  irreducible representation of dimension $2r-1$.  Let $H_2$ denote the image of $\SL_2$ under the representation $V_2\oplus  V_{2r-4}$.
The only dimensions of non-trivial invariant subspaces of $H_2$ are $3$ and $2r-3$.
We may therefore assume that $2\le p\le 2r-2$.

If $r= p_1^{e_1}\cdots p_l^{e_l}$, then $\prod_i \SL_{p_i}^{e_i}$ embeds in $\SL_r$ by the tensor product of natural representations, and thence into $\SO_{2r}$.   
Defining $H_1$ to be the image of this representation in $\SO_{2r}$, the restriction of the natural representation of $G$ to $H_1$
is the direct sum of  two irreducible factors of dimension $r$.  We may therefore assume, henceforth, that $k=r$.  

If $\cha K = 2$, there exists $2^e \in[r+1,2r]$, with $e\ge 2$.  Taking a tensor product of $e$ Frobenius twists of the natural representation of $\SL_2$, we obtain an irreducible, orthogonal representation of $\SL_2$ of degree $2^e$ and therefore an orthogonal representation of $\SL_2$ of dimension $2r$ with no $r$-dimensional inveriant subspaces. 

If $p\ge 3$, we use the construction of $H_3$ for $\SO_{2r-1}$ in Proposition~\ref{B irred}, and define $H_2$ for $\SO_{2r}$ to be the image of this group under the embedding $\SO_{2r-1}\subset \SO_{2r}$.
\end{proof}

\begin{theorem}
\label{generating sets}
Let $G$ be a simple algebraic group over an algebraically closed field $K$.  
\begin{enumerate}
\item[(1)] Except in the case that $G$ is of type $C_2$ and $\cha K=3$, $G$ contains a finite generating set of special subgroups. 
\item[(2)] If $G = \SL_n$ and $\cha K\neq 2$, then $G$ has a very special generating set.
\item[(3)] If $G$ is an orthogonal or symplectic group and $\cha K = 0$, then $G$ has a very special generating set.
\item[(4)] We can choose one element $H_i$ of a generating set so that $H_i\times H_i \not\prec G$.\end{enumerate}
\end{theorem}

\begin{proof}
As the statement does not depend on isogeny class, we assume $G$ is simply connected, except for orthogonal types, where we assume $G$ is a special orthogonal group.  
In each case, we give a finite set of special subgroups $H_i$ such that the only subgroup $H$ of $G$ satisfying $H_i\prec H$ for all $i$ is $G$ itself.
The $H_i$ are special and therefore semisimple, so if $H$ is a reductive and $H_i\prec H$ for all $i$, then the same is true of the derived group of $H^\circ$, so if $H$ is reductive, we may assume it is semisimple.

We consider each of the possible types:
\vskip 3pt
\noindent \textbf{Case $A_r$, $r\ge 1$.}  By Lemma~\ref{A irred} and in view of Corollary \ref{cor:pv}, there exists a special subgroup $H_1\subset G=\SL_{r+1}$ such that that any subgroup $H\subset G$ with  $H_1\prec H$ acts irreducibly on the natural representation $V$,
so it is reductive and may therefore be assumed to be semisimple.  If $r=1$, this implies $H=\SL_2$, so we may assume $r\ge 2$.
If $r\in\{2k-1,2k\}$, let $H_2 = \SL_2^k$ with the standard embedding to $G$.  By Theorem~\ref{degeneration}, the rank of any subgroup $H$ such that $H_2\prec H$ is at least $k$.

Let $\tilde H$ denote the simply connected covering group of $H$ and 
$$\tilde H = L_1\times \cdots\times L_l,$$
where the $L_i$ are simple.  If $r_i$ denotes the rank of $L_i$, then 
$r_1+\cdots+r_l \ge k$.  As $V$ is the exterior tensor product of almost faithful representations of the $H_i$,
$$r+1\ge   (r_1+1)\cdots (r_l+1).$$
The last two conditions imply that either $l=1$, or that $l=2$ and $r_1=1$, $r_2=k-1$. If $l=1$, $H$ is simple and must be $\SO_{r+1}$ or $\Sp_{r+1}$ (cf. \cite[Theorem 5.1 and Cor. 5.2]{lubeck} or \cite{liebeck}). If $l=2$ and $r_1=1$, the first factor of $H$ is $\SL_2$ with the natural module. So $\dim V$ is even and $r=2k-1$. Also the second factor has rank $k-1$ with an irrep of dimension $k$, so must be $\SL_k$ with its natural module. However $H_2=A_1^k$ does not embed in $\SL_2 \times \SL_k$, unless $k=2$. Therefore $r_2=1$, $r=3$, and after extension of scalars, $H = \SO_4$.

Let $H_3=\SL_3$, embedded in $\SL_{r+1}$ by any representation which is not self-dual.  Then there is no homomorphism $\SL_3\to \SO_{r+1}$ or $\SL_3\to \Sp_{r+1}$
consistent with the embeddings of $H_3$, $\SO_{r+1}$, and $\Sp_{r+1}$ in $\SL_{r+1}$, so the theorem holds in this case.

\vskip 3pt
\noindent \textbf{Case $B_r$, $r\ge 3$, $\cha K\neq 2$.}  
Let $H_1$, $H_2$, and $H_3$ be as in Proposition~\ref{B irred}, and let $H_4 = \SL_2^r$, which embeds as a subgroup of $G=\SO_{2r+1}$ since $G$ is always of the form $\SO_3\times \SO_4^{\frac{r-1}2}$
or $\SO_5\times \SO_4^{\frac{r-2}2}$.  By Proposition~\ref{B irred}, any group $H$ satisfying $H_i\prec H$
for all $i\in \{1,2,3\}$ acts  irreducibly on the natural representation of $G$, so it is reductive and we may therefore assume it is semisimple.
As $H_4\prec H$, $H$ must have rank $r$.  From the classification of equal rank semisimple subgroups of simple groups, this implies $H = G$.

\vskip 3pt
\noindent \textbf{Case $C_r$, $r\ge 2$, $\cha K\neq 3$ if $r=2$.}  
Let $H_1$ and $H_2$ be as in Proposition~\ref{C irred}, and let $H_3 = \SL_2^r$,  embedded as a subgroup of $G=\Sp_{2r}$ in the obvious way.  By Proposition~\ref{C irred}, any group $H$ satisyfing $H_1\prec H$ and $H_2\prec H$ acts  irreducibly on the natural representation of $G$, so $H$ is reductive and can be assumed semisimple.  As $H_3\prec H$, $H$ is of rank $r$.  If $\cha K\neq 2$, then  the classification of equal rank subgroups of simple groups implies $H = G$.
If $\cha K=2$, there is an additional possibility: $H=\SO_{2r}\subset \Sp_{2r}$.  However, $\SO_{2r}\subset \Sp_{2r}$ does not contain a subgroup for which the restriction of the natural representation of $\Sp_{2r}$ decomposes into $r$ irreducible $2$-dimensional factors, so this possibility is ruled out.

\vskip 3pt
\noindent \textbf{Case $D_r$, $r\ge 4$.}  
Let $H_1$ and $H_2$ be as in Proposition~\ref{D irred}.  If $r$ is even, let $H_3 = \SO_4^{r/2}$,  embedded as a subgroup of $G=\SO_{2r}$ in the obvious way; if $r$ is odd, let $H_3 = \SO_4^{(r-1)/2}$.  By Proposition~\ref{D irred}, any group $H$ satisfying $H_1\prec H$ and $H_2\prec H$ acts  irreducibly on the natural representation of $G$, so we may assume $H$ is semisimple of rank $\ge r-1$. If $H$ is of rank $r$, then $H = G$ follows from the classification of equal rank subgroups,
We may therefore assume $H$ is of rank $r-1$, which means $r$ is odd and therefore at least $5$.

If the universal covering group $\tilde H$ has simple factors of ranks $r_1\ge \cdots\ge r_l$, then 
$$r_1+\cdots+r_l = r-1$$
and
$$(2r_1)\cdots (2r_l) \le 2r,$$
since the minimum dimension of a self-dual representation of a simple group of rank $r_i$ is $2r_i$.
The only solutions for these two condition for $l\ge 2$ are $(r_1,r_2) = (2,1)$ (for $l=2$) and $r_1=r_2=r_3=1$ (for $l=3$).  However, in neither case is $n\ge 5$.  Therefore $H$ is simple, but a simple group of rank $r-1$ cannot have a self-dual representation of dimension $2r$ for $r\ge 5$.

\vskip 3pt
\noindent \textbf{Case $G$ exceptional.}  Suppose $H_1$ and $H_2$ are special subgroups of $G$ with the following properties.  The group $H_1$ is a maximal proper connected subgroup of $G$ which is semisimple, of rank equal to $\rk G$, and not isomorphic to a subgroup of any maximal proper connected subgroup of $G$ except conjugates of itself, while the group $H_2$ is not isomorphic to any subgroup of $H_1$.  If some proper subgroup $H$ of $G$ satisfies $H_1\prec H$, then $H_1$ is isomorphic to a quotient of $H^\circ$.  Therefore, the semisimple rank of $H^\circ$ must equal $\rk G$, and by \cite[Corollary 2 (ii)]{lieseitz}, that implies that $H^\circ$ is semisimple, so $H^\circ$ is conjugate to $H_1$.
However, $H_2$ is semisimple and does not embed in $H_1$, so we cannot have $H_2\prec H$.

Using \cite[Table 10.3]{lieseitz}, we find such pairs $(H_1,H_2)$ in all cases except  $E_7$.  
For $E_6$, take $H_1$ a subgroup of type $A_2^3$ and $H_2$ a subgroup of type $A_1\times A_4$
contained in a maximal subgroup of type $A_1\times A_5$.
For $E_8$, take $H_1$ a  subgroup  of type $A_1^8$ and $H_2$ a  subgroup of type $A_2^4$.  For $F_4$, take $H_1$ of type $A_2^2$ and $H_2$ of type $A_1^4$ (contained in a maximal subgroup of type $D_4$).
For $G_2$, take $H_1$ of type $A_1^2$ and $H_2$ of type $A_2$.
Except for case $F_4$, the dimension of $H_2$ is greater than that of $H_1$, so there is no embedding of $H_2$ in $H_1$.  For $F_4$, there is no embedding of a group of type $A_1^4$ in a group of type $A_2^2$ because there is no embedding of a group of type $A_1^2$ in a group of type $A_2$.

For $E_7$, we proceed somewhat differently, letting $H_1$ denote a special subgroup of $E_7$ of type $A_1^7$ (contained in a maximal subgroup of type $A_1\times D_6$) and letting $H_2$ be a special subgroup of type $A_2\times A_4$ (contained in a maximal subgroup of type $A_2\times A_5$).  Suppose a proper subgroup $H$ of $G$ satisfies  $H_1\prec H$ and $H_2\prec H$.  Without loss of generality, we may assume $H$ is connected.
Let $M$ be a maximal proper subgroup of $G$ containing $H_1$.  Then
$M$ contains the rank $7$ semisimple group $H_1$, so it must be of rank $7$ and therefore, by \cite[Table 10.3]{lieseitz}, must be a group whose identity component is semisimple of type $A_1\times D_6$, $A_7$, $A_2\times A_5$, $A_1^3\times D_4$,
or $A_1^7$.  However, only a group of type $A_1\times D_6$, $A_1^3\times D_4$, or $A_1^7$ can contain a group of type $A_1^7$, and no group of any of these types can contain a subgroup of type $A_2\times A_4$.
This finishes the proof of part (1).  

For part (2), we first use Proposition~\ref{A char not 2} to prove $H_i\prec H$ for all $i$ implies $H^\circ$ acts irreducibly on the natural representation of $\GL_n$.  We may therefore assume $n\ge 3$.
To the list of $H_i$, we add two additional very special subgroups of $\SL_n$, namely the image $H'$ of $\SL_2$ under the symplectic representation $V_0^{n-2}\oplus V_1$ and the image $H''$ of $\SL_2$ under the orthogonal representation $V_0^{n-3}\oplus V_2$.  Proposition~\ref{5.6} then proves that $H$ must be $\Sp_n$ with its natural representation or all of $\SL_n$.  The $\Sp_n$ can is ruled out by $H''\not\prec \Sp_n$.

For each case in part (3), we use very special subgroups associated to two different representations of $\SL_2$.  For type $B_r$, we use the representations $V_{2r}$ and $V_0^2\oplus V_{2r-2}$ to define very special subgroups $H'$ and $H''$ of $\SO_{2r+1}$.
Thus $H'$ is the principal $\SL_2$. which is known to be a maximal subgroup of $\SO_{2r+1}$
except when $r=3$ when it is contained in the maximal subgroup $G_2$  of $\SO_7$.  
In general $H''\not\prec H'$, and for $r=3$, $H''\not\prec G_2$, so in either case the statement of (3) holds.
For type $C_r$, we use the representations $V_{2r-1}$ (which gives the principal $\SL_2$ in $\Sp_{2r}$)
and $V_{2r-3}\oplus V_1$.  Since the principal $\SL_2$ is always maximal for symplectic groups, (3) holds.
For type $D_r$, we use the representations $V_0\oplus V_{2r-2}$, which gives the principal $\SL_2$, which we denote $H'$,
and $V_2\oplus V_{2r-4}$, which gives another very special subgroup, $H''$.  Now, $H'$ is contained in $\SO_{2r-1}$ be in no other group intermediate between $H'$ and $\SO_{2r}$, while $H''\not\prec \SO_{2r-1}$.

Part (4) holds because in each of the cases above except for case $A_r$ when $r$ is even, one of the $H_i$ has rank greater than half the rank of $G$.  For the remaining case, we use the fact that $\SL_{2k+1}$ does not contain any subgroup isogenous to $A_1^{2k}$.
\end{proof}

\section{Main Theorems} \label{main theorems}

In this section, we prove the main theorems of the paper.  We begin with a proposition.

\begin{proposition}  \label{crt}   Let $K_i, 1 \le i \le k$, be any finite separable extensions of an infinite field $K$, and let $G$ be any connected linear algebraic group over $K$.
For every element $x = (x_1,\ldots,x_k)\in \prod_{i=1}^k G(K_i)$, 
there exists a finitely generated $K$-domain $A$ with fraction field isomorphic to $K(t)$ as $K$-algebra
such that $x$ lies in the image of $G(A)\to\prod_{i=1}^k G(K_i)$.
\end{proposition}

\begin{proof}
By the theorem of the primitive element, each $K_i$ is isomorphic as $K$-algebra to $K[t]/\m_i$, where, as $K$ is infinite, the $\m_i$ can be chosen to be pairwise distinct maximal ideals.
Fix for each $K_i$ an element $x_i\in G(K_i)$.   Our goal is to find a polynomial $Q(t)\in K[t]$ not in any $\m_i$ and an element $x\in G(K[t,1/Q(t)])$ which reduces to $x_i$ under reduction (mod $\m_i$) for all $1\le i\le k$.

As $G$ is a connected linear algebraic group over $K$, it is rational as a $K$-variety.  Let $U\subset G$ denote a non-empty open $K$-subvariety of $G$
which is isomorphic to an open subvariety of $\A^n$.  As $K$ is infinite, $G(K)$ is Zariski-dense in $G$, so there exists a translate of $U$ by an element of $G(K)$ which contains all the closed points of $G$ in $\{x_1,\ldots,x_k\}$,
so without loss of generality, $G$ has this property.  We fix an open immersion $\iota\colon U\to \A^n$ over $K$.

By the Chinese Remainder Theorem, the natural homomorphism $K[t]\to \prod_i K[t]/\m_i$ is surjective, so there exist elements $A_1(t),\ldots,A_n(t)\in K[t]$
such that $A_j(t)$ (mod $\m_i$) gives the $j$th coordinate of $\iota(x_i)$.  In other words, the $n$-tuples $(A_1(t),\ldots,A_n(t))$ defines a morphism $\xi\colon \A^1\to \A^n$ which maps the closed point associated to $\m_i$ to $\iota(x_i)$ for all $i$.
Thus $\xi^{-1}(\iota(U))$ is an open subset of $\A^1$.  Its coordinate ring is therefore of the form $A=K[t,1/Q(t)]$, and the restriction of $\xi$ gives a morphism $\Spec A\to U$, proving the proposition.
\end{proof}

We remark that it follows immediately that every finitely generated free subgroup of 
$\prod_{i=1}^k G(K_i)$ lies in the image of $G(A)$ for some choice of $A$.

We can now finish the proofs of Theorems \ref{char 0} and \ref{char p}.    
Let $K$ be a global field.  By Corollary~\ref{special}, for every split simple algebraic group $G$ over $K$, and every special subgroup $H$ of $G$, $H(K)$ contains a strongly dense free subgroup isomorphic to $F_2$.  We first assume that $G$ is not of type $C_2$ if $\cha K=3$.  
By Theorem~\ref{generating sets}, there exist almost abelian algebraic groups $H_1,\ldots,H_s$ over $K$ such that 
$\{(H_1)_{\bar K},\ldots,(H_s)_{\bar K}\}$ is a generating set for 
$G_{\bar K}$.

By Lemma~\ref{A1}, we may choose
injective homomorphisms $f_i\colon F_2\to H_i(K)$  with strongly dense images $f_i(F_2) = \Gamma_i\subset H_i(K)\subset  G(\bar K)$ for $1\le i\le s$.

By Proposition~\ref{crt}, there exists a finitely generated $K$-domain $A$ with maximal ideals $\m_1,\ldots,\m_s$ and a homomorphism $F_2\to G(A)$ which specializes $\pmod {\m_i}$ to $f_i$ for $i=1,\ldots,s$.  Moreover, we may assume the fraction field of $A$ is $K(t)$.
Thus, for every non-abelian subgroup $\Delta\subset F_2$, $f_i(\Delta)$ has Zariski-closure $H_i$.  The Zariski closure of the image of $\Delta$ in $G(K(t))$ degenerates to each of the $H_i$, so by Theorem~\ref{degeneration}, it must be all of $G$, as claimed.

The same method works for $\Sp_4$ when $\cha K=3$, except that instead of a generating set of special subgroups of $\Sp_4$, we have a generating set consisting of
the image $H_1$ of $\SL_2\times \SL_2$ in $\Sp_4$ and the derived group $H_2$ of the stabilizer of a line in the natural representation.
By \cite[Appendix D]{BGGT2}, $H_2$ contains a strongly dense free subgroup $F_2$, and we define $L$, $A$, $f_i$, and so on, as before.  Since $H_1$ is a maximal connected subgroup of $G$, any connected subgroup $H$ which can degenerate
to $H_1$ is either $H_1$ or $G$, and $H_1$ cannot degenerate to $H_2$ because $H_1$ has an invariant $2$-dimensional subspace while $H_2$ does not.  Thus, if $G$ is simple then $G(K(t))$ contains a strongly dense free subgroup.

Now consider the general case when $G$ is semisimple.  There is no harm in assuming that $G$ is simply connected
(or has trivial center)  and so is a direct product.  By Goursat's lemma, we can assume that all simple factors are of the same type.
So say $G = \underbrace{J \times \ldots \times J}_k$ with $J$ simple.    
By Theorem~\ref{generating sets}, $J_{\bar K}$ has a generating set $\{(H_1)_{\bar K},\ldots,(H_s)_{\bar K}\}$,
and we may assume that $J$ cannot degenerate to any group isomorphic to $H_1\times H_1$.
In fact, 
for each $i\in \{1,\ldots,s\}$ and each $j\in \{1,\ldots,k\}$ we define an almost abelian $K$-subgroup $H_{i,j}$ of $J$ such that $(H_{i,j})_{\bar K}$ is isomorphic to $(H_i)_{\bar K}$ and such that each of the division algebras associated to any $H_{i,j}$ is ramified over some prime of $K$ which none of the others is ramified over.

By Theorem~\ref{generating sets} (4), we may assume that $J_{\bar K}$ does not have any subgroup isomorphic to $(H_1)_{\bar K}\times (H_1)_{\bar K}$.
We  fix for each $i,j$ a homomorphism $f_{i,j}\colon F_2\to H_{i,j}(K)\subset J(\bar K)$, and choose $A$, $\m_1,\ldots,\m_s$, and $\tilde f_j\colon F_2\to J(A)$ which reduces to $f_{i,j}$ modulo $\m_i$.  Thus each $\tilde f_j(F_2)$ is Zariski-dense in $J(K(t))$.  Defining $\tilde f = (\tilde f_1,\ldots,\tilde f_k)\colon F_2\to J^k(K(t))$, it follows that for all non-abelian subgroups $\Delta$ of $F_2$, the Zariski closure of $\tilde f(\Delta)$ maps surjectively to each factor $J$.
By Goursat's lemma to prove that $\tilde f(\Delta)$ is Zariski-dense, it suffices to prove that it is not contained in the graph of any isomorphism between any two factors of $J$.  

We may therefore assume that $k=2$ and for some $\Delta\subset F_2$, $\tilde f(\Delta)$ has Zariski closure isomorphic to $J$.  However the Zariski-closure of its (mod $\m_1$) reduction
$(f_{1,1},f_{1,2})(\Delta)$ in $J\times J$ 
is $H_{1,1}\times H_{1,2}$ since it maps onto $H_{1,1}$ and $H_{1,2}$, which are non-isomorphic simple algebraic groups.  This is impossible since $J_{\bar K}$ does not contain a subgroup isomorphic to $(H_{1,1})_{\bar K} \times (H_{1,2})_{\bar K}$.

Since every transcendental field of characteristic $0$ contains a subfield isomorphic to $\Q(t)$, taking $K=\Q$, this gives the first part of Theorem~\ref{char 0}.
Since every extension of $\F_p$ of transcendence degree $\ge 2$ contains a field isomorphic to $\F_p(s,t)$, taking $K=\F_p(s)$ gives the first part of Theorem~\ref{char p}.

To prove the second part of both theorems (the density of pairs generating a strongly dense subgroup), observe that this follows easily from the existence part. Indeed, if $F$ is a strongly dense subgroup, then the set of non-commuting pairs $(x,y) \in F \times F$ is itself Zariski-dense in $G \times G$. This is because  for each $x \in F\setminus \{1\}$ it contains $\{x\} \times (F\setminus C_F(x))$, whose Zariski-closure is $\{x\} \times G$.

\section{Strongly Dense Subgroups in Affine Groups}

Almost all the results about strongly dense subgroups have been in semisimple groups.   Suppose that
$G$ is not semisimple.  When can $G$ contain a strongly dense free nonabelian subgroup $\Gamma$?   

We note:

\begin{lemma}  Suppose that an algebraic group $G$ over an algebraically closed field $k$ contains a strongly dense
nonabelian free subgroup $\Gamma$.  Then $G$ can be topologically generated by $2$ elements, 
$G$ is connected,  and $G$ is perfect.
\end{lemma}

\begin{proof}
As $\Gamma$ contains rank $2$ free subgroups  which are Zariski-dense, the first assertion holds.
Suppose that $G$ is not perfect.  Then the Zariski closure of $\Gamma'$ is contained in $G'$ whence
$\Gamma'$ is not Zariski-dense.   If $G$ is not connected, then $G^{\circ} \cap \Gamma$ is not Zariski-dense
in $G$. 
\end{proof}

In particular, this implies that   $G/R_u(G)$ is semisimple.    We can always replace $G$ by $G/\Phi(G)$ where $\Phi(G)$
is the Frattini subgroup of $G$ and so assume that $R_u(G)$ is a completely reducible $G$-module.  
The condition that $G$ is $2$-generated imposes a limit on the multiplicities of the composition factors in $R_u(G)$ (in terms
of dimension and cohomology).   We do note that:

\begin{lemma}   Suppose that $R_u(G)$ is a simple $G$-module and $k$ is not algebraic over a finite
field.   Then $G$ is $2$-generated (topologically).
\end{lemma}

This follows easily from the fact that 
$$\dim H^1(G/R_u(G), R_u(G)) < \dim R_u(G)$$
  (which is an old result
from \cite{AG} for finite groups and the proof for algebraic groups is much easier and also follows from the same
inequality for finite groups of Lie type).  

We ask whether any connected perfect algebraic group over an algebraically closed field  that is
 topologically generated by two elements contains a strongly dense subgroup.

Here we show that certain affine groups do have this property.  We could extend this but content ourselves with
considering the affine groups $\ASL_n(K) = V.\SL(V)$ with $\dim V = n > 1$. 

\begin{theorem}  Let $K=\QQ(t)$ or $\FF_p(s,t)$.   Let $G =\ASL_n(K)$.   Then $G$ contains strongly dense free
subgroups.
\end{theorem}

\begin{proof}  If $n=2$, this is proved in \cite[Appendix D]{BGGT2} where $K$ is any field not algebraic over a finite
field.   So assume that $n > 2$.   We now choose elements $x, y \in  \ASL_n(\QQ[t])$ or $\ASL_n(\FF_p(s)[t])$
that specialize to strongly dense subgroups of $H_1, \ldots, H_r$ with $H_i \le \SL_n(\bar{K})$ as in the proof
of the main theorems for $\SL_n$.  

Let $H_0$ be a $2$-generated free strongly dense subgroup of $\ASL_2(K)$ naturally embedded in $\ASL_n(K)$.  
We pick $x, y$ that also specialize to a strongly dense subgroup of $H_0$.   Let $S = \langle x, y \rangle$ and let
$T$ be any nonabelian subgroup of $S$.   As we argued in the proof of our main result, we see that the Zariski closure
is either $G$ or is a complement to $V=R_u(G)$.  

The rational cohomology group $H^1(\SL_n,V)=0$ by \cite{Andersen} since if $\delta$ denotes the half sum of positive roots of $\SL_n$ and $\varpi_1$ is the highest weight of $V$, the weights $-\delta$ and $\varpi_1-\delta$ are not in the same Weyl group orbit (mod $p$).
Thus,
$T$ is contained a complement to $V$ in $G$ if and only if $T$ fixes a nonzero vector in the $n+1$ dimensional representation
of $G$ (we embed $G = \SL_{n+1}(K)$ contained in the stabilizer of a hyperplane).   Note that this is a closed condition
(this is equivalent to commuting with a rank $1$ idempotent whose kernel is the given hyperplane---the set
of such idempotents is the set of conjugates of a single such idempotent by $V$).     Since $H_0$ commutes
with no such idempotent,  it follows that $T$ does not either and so the Zariski closure of $T$ is $G$. 
\end{proof}

\section{Nonfree Strongly Dense Subgroups}

In this section, we give examples of finitely generated groups which are not free but can nevertheless be embedded in groups of the form $G(K)$ as strongly dense subgroups in the sense of Definition~\ref{strongly dense}.

Recall that a group $H$  is called \emph{residually free} if for every nontrivial element  $h\in H$,
there exists a free quotient $J$ of $H$ such that the image of  $h$ in $J$ is nontrivial.   

We now show that the methods of \cite{BGGT} and this paper can be used to prove that a large class of
groups (including surface groups of genus at least $2$) have strongly dense embeddings.

\begin{theorem}   
\label{frf}
Let $\Gamma$ be a finitely generated group satisfying the following conditions:
\begin{enumerate}
\item   $\mathrm{Hom}(\Gamma, G)$ is an irreducible variety for every simply connected simple algebraic group $G$.
\item   $\Gamma$ has trivial center and is residually free.
\end{enumerate}
If $K$ is an algebraically closed field of infinite transcendence degree and characteristic $p \ge 0$ and $G$ is a semisimple algebraic
group over $K$, then there exist strongly dense embeddings of $\Gamma$ into $G(K)$.
\end{theorem}

Since $\Gamma$ has trivial center, it suffices to prove the theorem in the simply connected case, so we assume henceforth that $G$ is simply connected.   We do not need infinite transcendence degree to make the argument work, but the degree needed with this argument grows linearly with $\dim G$.

Let $\Sigma_g$ be the fundamental group of a Riemann surface of genus $g$.   It is well known that
$\mathrm{Hom}(\Sigma_g, G)$  is an irreducible variety for any simply connected simple algebraic group $G$.
This is shown by Simpson \cite[Thm. 11.1]{Simpson} for the $\SL_n$ case (see also \cite{BCR}) and by
Liebeck and Shalev \cite[Cor.~1.11(ii)]{LS} for the general case (note that they only claim the irreducibility of the top dimensional component which is of dimension $(2g-1)\dim G$, but since $\mathrm{Hom}(\Sigma_g, G)$ is the preimage of the identity by an algebraic morphism, each component must be top dimensional). It is also well known that $\Sigma_g, g \ge 2$, is residually free (cf.  \cite[Cor. 2.2]{breuillard-gelander3}) and so:

\begin{corollary}  There exist strongly dense embeddings of $\Sigma_g, g \ge 2$ into $G(K)$ for $K$ an 
algebraically closed field of infinite transcendence degree in any characteristic and $G$ any semisimple group over $K$.
\end{corollary}  

We remark that Long, Reid, and Wolff \cite{LRW} use a similar strategy to show that generic Hitchin representations are strongly dense.

We begin with a lemma:

\begin{lemma}
\label{HH}
Let $G$ be a simply connected semisimple algebraic group over an algebraically closed field $K$. 
There exists a countable collection of proper closed subvarieties $Z_i$ of $G \times G$, each defined over the prime
subfield, such
that  $\cup_{i} Z_i$ is the set of all $(g_1,g_2)\in Z_i$, such that group $\langle g_1,g_2\rangle$ is not Zariski-dense in $G$.
\end{lemma}

\begin{proof}
We first assume that $G$ is simple and simply connected.   Let $k$ be the algebraic closure of the prime field. 
By \cite[Thm. 11.7]{gt},   there exists a finite set of irreducible rational $G$-modules (defined over $k$) so that no
positive dimensional closed subgroup of $G$ acts irreducibly on each of those modules.   The set of pairs acting reducibly
for each module is a proper closed subvariety.   

Now consider the proper subvarieties $X_m = \{(g_1, g_2) \in G \times G | |\langle g_1, g_2 \rangle | \le m\}$.  These are defined over the prime field.
The result now follows in this case by taking the $Z_i$ to be the finite set of subvarieties given by the modules together with the countably
many subvarieties $X_m$.

If $G$ is simply connected but not simple, then $G = G_1\times \cdots\times G_N$
for some simply connected simple groups $G_i$, and
every maximal subgroup of $G$ is either the pull-back of a maximal subgroup of some $G_i$ or the pull-back of the graph of a surjective endomorphism between the adjoint quotients of two factors, $G_i$ and $G_j$.  We have already dealt with the first class of subgroups.    Up to the action of $G$,  the maximal diagonal
subgroups correspond to compositions of outer automorphisms and Frobenius endomorphisms and in particular there are only countably many such
and each is defined over a finite extension of the prime field giving rise to countably many conjugacy classes of maximal closed diagonal subgroups $D_i$.
For each $D_i$, we consider the subvariety which is the closure of $\cup_{g \in G} (D_i \times D_i)^g$.  Clearly these are proper subvarieties (as generic elements
are not contained in any diagonal subgroup).  
\end{proof}

One can also show that the complement of the union of the subvarieties above is dense as long as $K$ is not algebraic over a finite field
(if $K$ is uncountable, this is clear).    An alternate proof of the previous result can be obtained by noting that there are only countably many
maximal proper closed subgroups (maximal in the category of closed subgroups) and they are all defined over the algebraic closure of the prime field. 
In characteristic $0$, one already knows that  the set of generating pairs is a nonempty open subset  (see \cite[Theorem 4.1]{BGGT}).  

We can now deduce Theorem~\ref{frf} from the analogous result on free groups.
\vskip 10pt

\begin{proof}
By condition (i), $\Hom(\Gamma,G)$ cannot be written as a countable union of proper closed subvarieties. So it is enough to prove that given any pair $\gamma_1,\gamma_2$ of non-commuting elements in $\Gamma$, and for each closed subvariety $Z_i$ from Lemma 8.3, the closed subvariety $W_{i,\gamma_1,\gamma_2}$ of $\Hom(\Gamma,G)$ made of those representations $\rho$ such that the pair $(\rho(\gamma_1),\rho(\gamma_2))$ belongs to $Z_i$ is proper. By \cite{BGGT}, there is a strongly dense free subgroup in $G$ given say by some injective homomorphism $\pi: F_2 \to G$. By condition (ii) there is a homomorphism $\phi:\Gamma \to F_2$ such that $\phi([\gamma_1,\gamma_2]) \neq 1$. The representation $\rho:=\pi \circ \phi$ is not in $W_{i,\gamma_1,\gamma_2}$, because $\langle \rho(\gamma_1),\rho(\gamma_2)\rangle$ is Zariski-dense. So $W_{i,\gamma_1,\gamma_2}$ is proper as desired.
\end{proof}

We can also extend this result (essentially via the  proof of \cite{BGGT} for free groups) to finitely generated groups $\Gamma$
such that all representation varieties are irreducible and  satisfy the Borel property (i.e. word maps are dominant 
-- see \cite{BL}).     This argument works aside from the case of groups involving $C_2$ in characteristic $3$ 
(just as the proof in \cite{BGGT} did -- that case was handled in \cite{BGGT2}).   

To give another application of Theorem~\ref{frf}, we introduce the following terminology.
A word in $F_d$ is $(N,l)$-\emph{friable} if it is a justaposition of at least $N$ non-empty words of length $\le l$ in pairwise distinct variables.  Recall that by the \emph{Baumslag double} of a word 
$w = w(x_1,\ldots,x_d)\in F_d$, we mean the one-relator group on $2d$-generators
$$\langle x_1,\ldots,x_d,y_1,\ldots,y_d\mid w(x_1,\ldots,x_d)w(y_1,\ldots,y_d)^{-1}\rangle.$$

\begin{theorem}
For all $l$ there exists $N$ such that if $w\in F_d$ is $(N,l)$-friable, and $\Gamma$ is the Baumslag double of $w$, then for every semisimple group $G$ over an algebraically closed field of infinite transcendence degree, $G(K)$ contains a strongly dense subgroup isomorphic to $\Gamma$.
\end{theorem}

\begin{proof}
Condition (i) holds if $N$ is large enough compared to $l$.
By \cite[Theorem~5(ii)]{larsen-shalev-tiep}, if $N$ is large enough in terms of $l$,
$$|\Hom(\Gamma,H)| = (1+o(1))|H|^{d-1}$$
for finite simple groups $H$.  However,
the character estimate for groups of Lie type used in the proof, namely  \cite[Thm~1.4]{glt} and \cite{gluck}, are both proved, in those papers, in the quasisimple case, and therefore 
also for groups of the form $\tilde H\cong G(\F_q)$, where $G$ is simple and simply connected.  
 
By Lang-Weil \cite{lang-weil}, letting $q\to \infty$, it follows that in positive characteristic $\Hom(\Gamma,G)$ is geometrically irreducible and of dimension $(d-1)\dim G$.  Therefore, the same is true in characteristic $0$.  

As long as $N>1$, no $(N,l)$-friable word is a non-trivial power of another word, so $\Gamma$ is  residually free \cite{baumslag}.  Therefore, the theorem follows from Theorem~\ref{frf}.

\end{proof}

Up to this point, the results of this section require that the transcendence degree of $K$ is large enough in terms of the dimension of $G$.  However, it is also possible to use the methods of sections 3--6 of this paper to prove certain $\dim G$-independent results.

\begin{theorem}
\label{trans.deg.one}
Let $K$ be a transcendental algebraically closed extension of $\Q$. Let $G$ be a classical group over $K$ and $\Gamma$ be a finitely generated group with $\Hom(\Gamma,G)$ irreducible and which admits a strongly dense embedding in $\SL_2(\bar \Q)$. Then $\Gamma$ admits a strongly dense embedding in $G(K)$.
\end{theorem}

\begin{proof}
Without loss of generality, we may assume that $G$ comes from a split group (which we also denote $G$) defined over $\Q$. 

In view of Theorem \ref{generating sets} (2) and (3), there exists a very special generating set  $\{H_i\}$ for $G$ and a finite collection of representations $f_i\colon \Gamma\to H_i(\bar\Q)\subset  \SL_n(\bar\Q)$ such that for every non-abelian subgroup $\Delta$ of $\Gamma$,  $f_i(\Delta)$ is Zariski-dense in $H_i(K)$.
The $f_i$ define points on the variety $\Hom(\Gamma,G)$, which is irreducible.

For any finite set $S$ of points on an irreducible quasi-projective variety $V$, there exists an irreducible curve containing $S$.  Indeed, one can blow up the points in $S$, embed the resulting variety $\tilde V$ in some projective space, and use Bertini's theorem to choose a linear subspace of codimension $\dim V-1$ intersecting $\tilde V$ in an irreducible curve whose image in $V$  contains $S$.  Applying this to the points on $V=\Hom(\Gamma,G)$ corresponding to the $f_i$, we obtain an irreducible affine curve with coordinate ring $A$. The universal $G$-representation of $\Gamma$ over $V$ specializes to a homomorphism $\Gamma\to G(A)$ which further specializes to each of the $f_i$ and is therefore injective.  If $K_0$ is the field of fractions of $A$, as the $\{H_i\}$ form a generating set, $\Gamma\to G(K_0)$ is strongly dense.  However, $K_0$ is of transcendence degree $1$, so it embeds in $K$.
\end{proof}

\begin{corollary}
\label{surface char 0}
If $K$ be a transcendental algebraically closed extension of $\Q$ and $G$ is a classical group over $K$, then for each $g \ge 2$, $G(K)$ contains a strongly dense subgroup isomorphic to a surface group of genus $g$.
\end{corollary}

\begin{proof}Every Riemann surface of genus $g\ge 2$ can be realized as a quotient of the upper half-plane by a subgroup of $\PSL_2(\R)$ isomorphic to the surface group  $\Gamma$.  It is also well known that the injective homomorphism $\Gamma\to \PSL_2(\R)$ lifts to a (necessarily injective) homomorphism $\Gamma\to \SL_2(\R)$ (see, e.g., \cite{AAS} for a short argument).  The subset of $\Hom(\Gamma, \SL_2(\R))$ such that the map $\Gamma\to \PSL_2(\R)$ is injective with discrete image is a non-empty open subset of $\Hom(\Gamma,\SL_2(\R))$ (\cite{Weil}).  Therefore, there exists an injective homomorphism $\phi\colon \Gamma\to \SL_2(\bar{\Q})$. Its image is necessarily strongly dense, because proper algebraic subgroups of $\SL_2$ are virtually solvable and $\Gamma$ has no non-abelian virtually solvable subgroup. So we may apply the previous theorem.
\end{proof}

We remark that, at least when $G=\SL_n$, the group of $G$-representations of a surface group is rational \cite{BCR}, so one might hope to find a rational curve containing the points corresponding to the $f_i$.  If this can be done, we can dispense with the assumption that $K$ is algebraically closed.

\begin{theorem}
Let $p > 2$ be a prime and $n,g\ge 2$ integers.
Let $K$ be an algebraically closed field of transcendence degree $10$ over $\F_p$.
Then $\SL_n(K)$ has a strongly dense subgroup isomorphic to the surface group of genus $g$.
\end{theorem}

\begin{proof}
It suffices to prove the theorem for $g=2$.  Let $\Gamma = \Sigma_2$, and let $A$ denote the coordinate ring of the affine variety $\Hom(\Gamma,\SL_2)$.  This variety is irreducible.  
The coordinate ring $A$ of this variety is of dimension $9$.   Let $L$ be the fraction field of $A$,
and let $\phi\colon \Gamma\to \SL_2(L)$ denote the composition of the universal $\SL_2$-representation of $\Gamma$ over $A$ with the embedding $\SL_2(A)\subset \SL_2(L)$.

For $1\neq \gamma\in \Gamma$, then there exists a homomorphism $\Gamma\to F_d$ sending $\gamma\mapsto \bar\gamma\neq 1$, and there exists a homomorphism $F_d\to \SL_2(K)$ which is injective.  Therefore, the image of $\gamma$ in $\SL_2(L)$ is not $1$, so $\phi$ is injective.

If $\Delta\subset \Gamma$ is a non-abelian subgroup, we would like to prove that $\phi(\Delta)$
is Zariski-dense in $\SL_2(L)$.  If $\Delta$ is a non-abelian subgroup of $\Gamma$ and $\Delta_1$ is a subgroup of finite index, then $\Delta_1$ is either a surface group or a non-abelian free group, so the commutator subgroup of its commutator subgroup is non-trivial.  If $\gamma$ is a non-trivial element in this group, then $\phi(\gamma)\neq 1$, so $\phi(\Delta_1)$ cannot be contained in a Borel subgroup of $\SL_2(\bar K)$.  By classification of closed subgroups of $\SL_2$, every proper subgroup has a finite index subgroup whose second commutator is trivial.  Therefore, $\phi(\Delta)$ is indeed Zariski-dense in $\SL_2$.

Now we proceed as in the proof of Theorem~\ref{surface char 0}. using part (2) of Theorem~\ref{generating sets} to show that $\SL_n$ has a very special generating set and then using the connectedness of $\Hom(\Gamma,\SL_n)$ in positive characteristic to deduce that for an algebraically closed field of transcendence degree $10$ over $\F_p$, there is a strongly dense subgroup isomorphic to $\Gamma$ and therefore a strongly dense subgroup isomorphic to $\Sigma_g$ for each $g\ge 2$.
\end{proof}

With a little more work, we can use the quotient of $\Hom(\Gamma,\SL_2)$ by the action of $\SL_2$ to reduce the transcendence degree to $7$.  In \cite{FLSS}, it is proved that for $p\ge 5$, there are faithful representations of $\Gamma$ in $\PGL_2$ over fields of transcendence degree $2$.  It seems possible that one could prove the same result for $\SL_2$ and use this to reduce transcendence degree to $3$.

\setcounter{tocdepth}{1}

\end{document}